\documentclass[12pt,
psamsfonts]{amsart}

\usepackage{amssymb,amsfonts,amsmath}
\usepackage[all,arc]{xy}
\usepackage{enumerate}
\usepackage{mathrsfs}
\usepackage{fullpage}
\usepackage{xspace}
\usepackage[margin=1.0in]{geometry}

\newtheorem{thm}{Theorem}[section]
\newtheorem{cor}[thm]{Corollary}
\newtheorem{prop}[thm]{Proposition}
\newtheorem{lem}[thm]{Lemma}

\theoremstyle{definition}
\newtheorem{defn}[thm]{Definition}

\newtheorem{exmp}[thm]{Example}

\newtheorem{notn}[thm]{Notation}

\newtheorem{conv}[thm]{Convention}
\newtheorem{ass}[thm]{Assumption}

\theoremstyle{remark}
\newtheorem{rem}[thm]{Remark}

\newtheorem{notation}[thm]{Notation}

\newcommand{\Z}{\mathbb{Z}\xspace}
\newcommand{\Q}{\mathbb{Q}\xspace}
\DeclareMathOperator{\Spec}{Spec}
\DeclareMathOperator{\res}{res}
\DeclareMathOperator{\Tr}{Tr}
\DeclareMathOperator{\ord}{ord}

\DeclareMathOperator{\dv}{div}

\DeclareMathOperator{\img}{Im}

\DeclareMathOperator{\Reg}{Reg}
\DeclareMathOperator{\Cor}{Cor}
\DeclareMathOperator{\Ac}{at}
\DeclareMathOperator{\supp}{supp}
\DeclareMathOperator{\ch}{char}
\DeclareMathOperator{\Res}{Res}
\DeclareMathOperator{\Pic}{Pic}
\DeclareMathOperator{\HI}{HI}
\DeclareMathOperator{\Nis}{Nis}
\DeclareMathOperator{\NST}{NST}

\makeatletter
\let\c@equation\c@thm
\makeatother
\numberwithin{equation}{section}

\newcommand{\textlatin }

\title{The local symbol complex of a Reciprocity Functor}
\author{Evangelia Gazaki}
\email{valiagaz@math.uchicago.edu}
\address{Department of Mathematics\\ University of Chicago\\ 5734 University Ave. Chicago, Illinois 60637}
\begin{document}
\maketitle
\begin{abstract} For a reciprocity functor $\mathcal{M}$ we consider the local symbol complex
\[(\mathcal{M}\otimes^{M}\mathbb{G}_{m})(\eta_{C})\to\bigoplus_{P\in C}\mathcal{M}(k)\to\mathcal{M}(k),\] where $C$ is a smooth complete curve over an algebraically closed field $k$ with generic point $\eta_{C}$ and $\otimes^{M}$ is the product of Mackey functors. We prove that if $\mathcal{M}$ satisfies certain assumptions, then the homology of this complex is isomorphic to the $K$-group of reciprocity functors $T(\mathcal{M},\underline{CH}_{0}(C)^{0})(\Spec k)$.
\end{abstract}
\section{introduction} Let $F$ be a perfect field.
We consider the category $\mathcal{E}_{F}$  of finitely generated field extensions of $F$. In \cite{IR}, F. Ivorra and K. R\"{u}lling  created a theory of reciprocity functors. A reciprocity functor is a presheaf with transfers in the category $\Reg^{\leq 1}$ of regular schemes of dimension at most one over some field $k\in\mathcal{E}_{F}$ that satisfies various properties.\\
Some examples of reciprocity functors include commutative algebraic groups, homotopy invariant Nisnevich sheaves with transfers, K\"{a}hler differentials. Moreover, if $\mathcal{M}_{1},\dots,\mathcal{M}_{r}$ are reciprocity functors, Ivorra and R\"{u}lling construct a $K$-group $T(\mathcal{M}_{1},\dots,\mathcal{M}_{r})$  which is itself a reciprocity functor. \\
One of the crucial properties of a reciprocity functor $\mathcal{M}$ is that it has local symbols. Namely, if $C$ is a smooth, complete and geometrically connected curve over some field  $k\in\mathcal{E}_{F}$ with generic point $\eta$, then at each closed point $P\in C$ there is a local symbol assignment
\[(.;.)_{P}:\mathcal{M}(\eta)\times\mathbb{G}_{m}(\eta)\to\mathcal{M}(k),\] satisfying three characterizing properties, one of which is a reciprocity relation $\displaystyle\sum_{P\in C}(g;f)_{P}=0$, for every $g\in\mathcal{M}(\eta)$ and $f\in\mathbb{G}_{m}(\eta)$. We note here that if $G$ is a commutative algebraic group over an algebraically closed field $k$, then the local symbol of $G$ coincides with the local symbol constructed by Rosenlicht-Serre in \cite{Ser}.
The reciprocity relation induces a local symbol complex $(\underline{\underline{C}})$ \[(\mathcal{M}\bigotimes^{M}\mathbb{G}_{m})(\eta)\stackrel{((.;.)_{P})_{P\in C}}{
\longrightarrow}\bigoplus_{P\in C}\mathcal{M}(k)
\stackrel{\sum_{P}}{\longrightarrow}\mathcal{M}(k),\] where by $\otimes^{M}$ we denote the product of Mackey functors (see def. \ref{Mac}). The main goal of this article is to give a description of the homology $H(\underline{\underline{C}})$ of the above complex in terms of $K$-groups of reciprocity functors. Our computations work well for curves $C$ over an algebraically closed field $k$. In the last section we describe some special cases where the method could be refined to include non-algebraically closed base fields. To obtain a concrete result, we need to impose two conditions on the reciprocity functor $T(\mathcal{M},\underline{CH}_{0}(C)^{0})$ (see assumptions \ref{AS1}, \ref{AS2}).  In section 3 we prove the following theorem.
\begin{thm}\label{intro} Let $C$ be a smooth, complete curve over an algebraically closed field $k$. Let $\mathcal{M}$ be a reciprocity functor such  that the $K$-group of reciprocity functors $T(\mathcal{M},\underline{CH}_{0}(C)^{0})$ satisfies the assumptions \ref{AS1} and \ref{AS2}. Then the homology of the local symbol complex $(\underline{\underline{C}})$ is canonically isomorphic to the $K$-group $T(\mathcal{M},\underline{CH}_{0}(C)^{0})(\Spec k)$.
\end{thm} Here $\underline{CH}_{0}(C)^{0}$ is a reciprocity functor that is identified with the Jacobian variety $J$ of $C$. \\
In section 4 we give some examples of reciprocity functors that satisfy the two assumptions. In particular, we prove the following theorem.
\begin{thm} Let $\mathcal{F}_{1},\dots,\mathcal{F}_{r}$ be homotopy invariant Nisnevich sheaves with transfers, and consider the reciprocity functor $\mathcal{M}=T(\mathcal{F}_{1},\dots,\mathcal{F}_{r})$. Let $C$ be a smooth, complete curve over an algebraically closed field $k$. Then there is an isomorphism
\[H(\underline{\underline{C}})\simeq T(\mathcal{F}_{1},\dots,\mathcal{F}_{r},\underline{CH}_{0}(C)^{0})(\Spec k).\] In particular, if $G_{1},\dots,G_{r}$ are semi-abelian varieties over $k$, then we obtain an isomorphism
\[H(\underline{\underline{C}})\simeq T(G_{1},\dots,G_{r},\underline{CH}_{0}(C)^{0})(\Spec k)\simeq
K(k;G_{1},\dots,G_{r},\underline{CH}_{0}(C)^{0}),\] where $K(k;G_{1},\dots,G_{r},\underline{CH}_{0}(C)^{0})$ is the Somekawa $K$-group attached to $G_{1},\dots,G_{r}$.
\end{thm}
Another case where the assumptions of theorem \ref{intro} are satisfied is when $\mathcal{M}=T(\mathcal{M}_{1},\dots,\mathcal{M}_{r})$ such that $\mathcal{M}_{i}=\mathbb{G}_{a}$ for some $i\in\{1,\dots,r\}$. Using the main result of \cite{RuYa} together with theorem 5.4.7. in \cite{IR}, we obtain the following corollary.
\begin{cor} Let $\mathcal{M}_{1},\dots,\mathcal{M}_{r}$ be reciprocity functors. Let $\mathcal{M}=T(\mathbb{G}_{a},\mathcal{M}_{1},\dots,\mathcal{M}_{r})$. Then for any smooth complete curve $C$ over $k$, $H(\underline{\underline{C}})=0$. In particular, if $\ch k=0$, the complex
$\Omega_{k(C)}^{n+1}\stackrel{\Res_{P}}{\longrightarrow}\bigoplus_{P\in C}\Omega^{n}_{k}
\stackrel{\sum_{P}}{\longrightarrow}\Omega_{k}^{n}$ is exact.
\end{cor}
The idea for theorem \ref{intro} stems from the special case when $\mathcal{M}=\mathbb{G}_{m}$. In this case the local symbol
$k(C)^{\times}\otimes^{M}k(C)^{\times}\stackrel{(.;.)_{P}}{\longrightarrow} k^{\times}$  at a closed point $P\in C$ factors through the group $T(\mathbb{G}_{m},\mathbb{G}_{m})(\eta_{C})$. By a theorem in \cite{IR} this group is isomorphic to the usual Milnor $K$-group $K_{2}^{M}(k(C))$ and we recover the Milnor complex
\[K_{2}^{M}(k(C))\to\bigoplus_{P\in C}k^{\times}\stackrel{\sum_{P}}{\longrightarrow}k^{\times}.\]
This complex was studied by M. Somekawa in \cite{Som} and R. Akhtar in \cite{Ak}. Using different methods, they both prove that the homology of the above complex is isomorphic to the Somekawa $K$-group $K(k;\mathbb{G}_{m},\underline{CH}_{0}(C)^{0})$. This group turns out to be isomorphic to the group $T(\mathbb{G}_{m},\underline{CH}_{0}(C)^{0})(\Spec k)$. (by \cite{IR}, theorem 5.1.8. and \cite{KY}, theorem 11.14). A similar result was proved by T. Hiranouchi in \cite{Hir} for his Somekawa-type additive $K$-groups. Our method to prove theorem \ref{intro} is similar to the method used by R. Akhtar and T. Hiranouchi.
\begin{notation} For a smooth connected variety $X$ over $k\in\mathcal{E}_{F}$, we denote by $k(X)$ the function field of $X$.
Let $C$ be a smooth complete curve over $k\in\mathcal{E}_{F}$ and $P\in C$ a closed point. We write $\ord_{P}$ for the normalized discrete valuation on $k(C)$ defined by the point $P$ and for an integer $n\geq 1$, we put
$U_{C,P}^{(n)}=\{f\in k(C)^{\times}:\ord_{P}(1-f)\geq n\}$.
\end{notation}
\vspace{5pt}
\textbf{Acknowledgement:} I would like to express my great gratitude to my advisor, Professor Kazuya Kato, for his kind guidance throughout the process of obtaining the current result. Moreover, I would like to thank very heartily Professor Kay R\"{u}lling whose advice and feedback was very essential for the completion of this article. K. R\"{u}lling provided me with great suggestions which improved very significantly this paper. He also shared his great expertise about reciprocity functors with me. Finally, I would like to deeply thank my referee for pointing out my mistakes and giving very useful suggestions to improve the paper. In particular, the proof of lemma \ref{Ga} was the referee's suggestion.
\vspace{5pt}
\section{Review of Definitions}
\subsection{Reciprocity Functors} Let $\Reg^{\leq 1}$ be the category with objects regular $F$-schemes of dimension at most one which are separated and of finite type over some $k\in\mathcal{E}_{F}$. Let $\Reg^{\leq 1}\Cor$ be the category with the same objects as $\Reg^{\leq 1}$ and with morphisms finite correspondences. A reciprocity functor $\mathcal{M}$ is a presheaf of abelian groups on $\Reg^{\leq 1}\Cor$ which satisfies various properties. Here we only  recall those properties that we will need later in the paper.
\begin{notation} Let $\mathcal{M}$ be a reciprocity functor. For $k\in\mathcal{E}_{F}$ we will write $\mathcal{M}(k):=\mathcal{M}(\Spec k)$.
\end{notation}
Let $E/k$ be a finite extension of fields in $\mathcal{E}_{F}$. The morphism $\Spec E\to\Spec k$ induces a pull-back map $\mathcal{M}(k)\to\mathcal{M}(E)$, which we call restriction and will denote by  $\res_{E/k}$. Moreover, there is a finite correspondence $\Spec k\to\Spec E$ which induces a push-forward $\mathcal{M}(E)\to\mathcal{M}(k)$, which we will call the trace and denote it by
$\Tr_{E/k}$. \\
\underline{Injectivity:} Let $C$ be a smooth, complete curve over $k\in\mathcal{E}_{F}$. Each open set $U\subset C$ induces a pull-back map $\mathcal{M}(C)\to\mathcal{M}(U)$ that is required to be injective. Additionally, if $\eta_{C}$ is the generic point of $C$, we have an isomorphism
\[\lim\limits_{\longrightarrow}\mathcal{M}(U)\stackrel{\simeq}{\longrightarrow}
\mathcal{M}(\eta_{C}),\] where the limit extends over all open subsets $U\subset C$. \\
\underline{Specialization and Trace maps:} Let $P\in C$ be a closed point. For each open $U\subset C$ with $P\in U$, the closed immersion $P\hookrightarrow U$ induces $\mathcal{M}(U)\to\mathcal{M}(P)$. We consider the stalk $\mathcal{M}_{C,P}=\lim\limits_{\longrightarrow}\mathcal{M}(U)$, where the limit extends over all open $U\subset C$ with $P\in U$. The above morphisms induce a specialization map
\[s_{P}:\mathcal{M}_{C,P}\to\mathcal{M}(P).\]
Moreover, for every closed point $P\in C$ we obtain a Trace map, which we will denote by \[\Tr_{P/k}:\mathcal{M}(P)\to\mathcal{M}(k).\]
\subsection{The modulus condition and local symbols} Let $\mathcal{M}$ be a reciprocity functor. Let $C$ be a smooth, projective and geometrically connected curve over $k\in\mathcal{E}_{F}$. The definition of a reciprocity functor imposes the existence for each section $g\in\mathcal{M}(\eta_{C})$ of a modulus $\mathfrak{m}$ corresponding to $g$. The modulus $\mathfrak{m}$ is an effective divisor $\mathfrak{m}=\sum_{P\in S}n_{P}P$ on $C$, where $S$ is a closed subset of $C$, such that $g\in\mathcal{M}_{C,P}$, for every $P\not\in S$ and for every function $f\in k(C)^{\times}$ with $\displaystyle f\in\bigcap_{P\in S} U_{C,P}^{(n_{P})}$, it holds
\[\sum_{P\in C\setminus S}\ord_{P}(f)\Tr_{P/k}(s_{P}(g))=0.\]
\begin{notation} Let $f\in k(C)^{\times}$ be such that $\displaystyle f\in\bigcap_{P\in S} U_{C,P}^{(n_{P})}$. Then we will write $f\equiv 1$ mod $\mathfrak{m}$.
\end{notation}
The modulus condition on $\mathcal{M}$ is equivalent to the existence, for each closed point $P\in C$, of a bi-additive pairing called the local symbol at $P$
\begin{eqnarray*}(.;.)_{P}:\mathcal{M}(\eta_{C})\times\mathbb{G}_{m}(\eta_{C})\to
\mathcal{M}(k)
\end{eqnarray*} which satisfies the following three characterizing properties:
\begin{enumerate}\item $(g;f)_{P}=0$, for $f\in U_{C,P}^{(n_{P})}$, where $\displaystyle\mathfrak{m}=\sum_{P\in S}n_{P}P$ is a modulus corresponding to $g$.
\item $(g;f)_{P}=\ord_{P}(f)\Tr_{P/k}(s_{P}(g))$, for all $g\in\mathcal{M}_{C,P}$ and $f\in k(C)^{\times}$.
\item $\displaystyle\sum_{P\in C}(g;f)_{P}=0$, for every $g\in \mathcal{M}(\eta_{C})$ and $f\in k(C)^{\times}$.
\end{enumerate}
The proof of existence and uniqueness of this local symbol is along the lines of Prop.1 Chapter III in \cite{Ser}.
In this paper we will use the precise definition of $(g;f)_{P}$, for $g\in\mathcal{M}(\eta_{C})$ and $f\in k(C)^{\times}$, so we review it here.\\
\underline{Case 1:} If $g\in\mathcal{M}_{C,P}$, property (2) forces us to define $(g;f)_{P}=\ord_{P}(f)\Tr_{P/k}(s_{P}(g))$. \\
\underline{Case 2:} Let $P\in S$. Using the weak approximation theorem for valuations, we consider an auxiliary function $f_{P}$ for $f$ at $P$, i.e. a function $f_{P}\in k(C)^{\times}$ such that $f_{P}\in U_{C,P'}^{(n_{P'})}$ at every $P'\in S$, $P'\neq P$ and $f/f_{P}\in U_{C,P}^{(n_{P})}$. Then we define
\[(g;f)_{P}=-\sum_{Q\not\in S}\ord_{Q}(f_{P})\Tr_{Q/k}(s_{Q}(g)).\]
\\
Using the local symbol, one can define for each closed point $P\in C$, $Fil_{P}^{0}\mathcal{M}(\eta_{C}):=\mathcal{M}_{C,P}$ and for $r\geq 1$
\[Fil_{P}^{r}\mathcal{M}(\eta_{C}):=\{g\in\mathcal{M}(\eta_{C}):(g;f)_{P}=0,\;{\rm for}\;{\rm all}\;f\in U_{C,P}^{(r)}\}.\]
Then $\{Fil_{P}^{r}\}_{r\geq 0}$ form an increasing and exhaustive filtration of $\mathcal{M}(\eta_{C})$.\\
The reciprocity functors $\mathcal{M}$ for which there exists an integer $n\geq 0$ such that it holds $\mathcal{M}(\eta_{C})=Fil_{P}^{n}\mathcal{M}(\eta_{C})$, for every smooth complete and geometrically connected curve $C$ and every closed point $P\in C$, form a full subcategory of $RF$, which is denoted by $RF_{n}$. (see def. 1.5.7. in \cite{IR}).
\vspace{2pt}
\subsection{$K$-group of Reciprocity Functors} Let $\mathcal{M}_{1},\dots,\mathcal{M}_{n}$ be reciprocity functors. The  $K$-group of reciprocity functors $T(\mathcal{M}_{1},\dots,\mathcal{M}_{n})$ is itself a reciprocity functor that satisfies various properties. (4.2.4. in \cite{IR}). We will not need the precise definition of $T(\mathcal{M}_{1},\dots,\mathcal{M}_{n})$, but only the following properties.
\begin{enumerate}[(a)] \item For $k\in\mathcal{E}_{F}$, the group
$T(\mathcal{M}_{1},\dots,\mathcal{M}_{n})(k)$ is a quotient of $(\displaystyle\mathcal{M}_{1}\bigotimes^{M}\dots\bigotimes^{M}\mathcal{M}_{n})(k)$, where by $\displaystyle\bigotimes^{M}$ we denote the product of Mackey functors (see def. \ref{Mac}).
The group $T(\mathcal{M}_{1},\dots,\mathcal{M}_{n})(k)$ is generated by elements of the form $\Tr_{k'/k}(x_{1}\otimes\dots\otimes x_{n})$, with $x_{i}\in\mathcal{M}_{i}(k')$, where $k'/k$ is any finite extension.
\item Let $C$ be a smooth, complete and geometrically connected curve over $L\in\mathcal{E}_{k}$ and let $P\in C$ be a closed point. Let $g_{i}\in \mathcal{M}_{i}(\eta_{C})$. Then
\begin{enumerate}[(i)]\item If for some $r\geq 0$ we have $g_{i}\in Fil_{P}^{r}\mathcal{M}_{i}(\eta_{C})$ for $i=1,\dots,n$, then $g_{1}\otimes\dots\otimes g_{n}\in Fil_{P}^{r}T(\mathcal{M}_{1},\dots,\mathcal{M}_{n})(\eta_{C})$. Moreover, if the element $g_{i}$ has modulus $\displaystyle\mathfrak{m}_{i}=\sum_{P\in S_{i}}n_{P}^{i}P$, for $i=1,\dots,n$, then $\displaystyle\mathfrak{m}=\sum_{P\in\cup S_{i}}\max_{1\leq i\leq n}\{n_{P}^{i}\}P$ is a modulus for $g_{1}\otimes\dots\otimes g_{n}$.
\item If $g_{i}\in Fil_{P}^{0}\mathcal{M}_{i}(\eta_{C})$, for $i=1,\dots,n$, then we have an equality \[s_{P}(g_{1}\otimes\dots\otimes g_{n})=s_{P}(g_{1})\otimes\dots\otimes s_{P}(g_{n}).\]
\end{enumerate}
\end{enumerate}
\vspace{2pt}
\subsection{Examples} Some examples of reciprocity functors include constant reciprocity functors, commutative algebraic groups, homotopy invariant Nisnevich sheaves with transfers. For an explicit description of each of these examples we refer to section 2 in \cite{IR}. The following example is of particular interest to us. \\
Let $X$ be a smooth projective variety over $k\in\mathcal{E}_{F}$. Then there is a reciprocity functor  $\underline{CH}_{0}(X)$ such that for any scheme $U\in \Reg^{\leq 1}$ over $k$ we have
\[\underline{CH}_{0}(X)(U)=CH_{0}(X\times_{k} k(U)).\] Since we assumed $X$ is projective, the degree map $CH_{0}(X)\to\Z$ induces a map of reciprocity functors
$\underline{CH}_{0}(X)\to\Z$ whose kernel will be denoted by $\underline{CH}_{0}(X)^{0}$. Both $\underline{CH}_{0}(X)$ and $\underline{CH}_{0}(X)^{0}$ are in $RF_{0}$.
\begin{rem} We note here that if $X$ has a $k$-rational point, then we have a decomposition of reciprocity functors $\underline{CH}_{0}(X)\simeq\underline{CH}_{0}(X)^{0}\oplus\Z$, where $\Z$ is the constant reciprocity functor. Moreover, if $\mathcal{M}_{1},\dots,\mathcal{M}_{r}$ are reciprocity functors, then by corollary 4.2.5. (2) in \cite{IR} we have a decomposition
\[T(\underline{CH}_{0}(X),\mathcal{M}_{1},\dots,\mathcal{M}_{r})\simeq
T(\underline{CH}_{0}(X)^{0},\mathcal{M}_{1},\dots,\mathcal{M}_{r})\oplus
T(\Z,\mathcal{M}_{1},\dots,\mathcal{M}_{r}).\]
\end{rem}
\underline{Relation to Milnor $K$-theory and K\"{a}hler differentials:} If we consider the reciprocity functor $T(\mathbb{G}_{m}^{\times n}):=T(\mathbb{G}_{m},\dots,\mathbb{G}_{m})$ attached to $n$ copies of $\mathbb{G}_{m}$, then for every $k\in\mathcal{E}_{F}$ the group  $T(\mathbb{G}_{m}^{\times n})(k)$ is isomorphic to the usual Milnor $K$-group $K_{n}^{M}(k)$ (theorem 5.3.3. in \cite{IR}). \\
Moreover, if $k$ is of characteristic zero, then the group $T(\mathbb{G}_{a},\mathbb{G}_{m}^{\times n-1})(k)$, $n\geq 1$, is isomorphic to the group of K\"{a}hler differentials $\Omega^{n-1}_{k/\Z}$ (theorem 5.4.7 in \cite{IR}).
\vspace{3pt}
\section{The homology of the complex}
\begin{conv} From now on, unless otherwise mentioned, we will be working over an algebraically closed  base field $k\in\mathcal{E}_{F}$.
\end{conv}
Let $\mathcal{M}$ be a reciprocity functor.
Let $C$ be a smooth complete curve over $k$ with generic point $\eta_{C}$. At each closed point $P\in C$ we have a local symbol $(.;.)_{P}$. We will denote by $(.;.)_{C}$ the collection of all symbols $\{(.;.)_{P}\}_{P\in C}$, namely
\[(.;.)_{C}:\mathcal{M}(\eta_{C})\otimes\mathbb{G}_{m}(\eta_{C})\to
\bigoplus_{P\in C}\mathcal{M}(k).\]
We note here that a reciprocity functor $\mathcal{M}$ is also a Mackey functor. In what follows, we will need the definition of the product of Mackey functors $\mathcal{M}_{1},\dots,\mathcal{M}_{r}$, evaluated at a finitely generated extension $L$ of $k$. We review this definition here.
\begin{defn}\label{Mac} Let $\mathcal{M}_{1},\dots,\mathcal{M}_{r}$ be Mackey functors over $k$. Let $L$ be a finitely generated extension of $k$. Then,
\[(\mathcal{M}_{1}\bigotimes^{M}\dots\bigotimes^{M}\mathcal{M}_{r})(L):=
(\bigoplus_{L'/L} \mathcal{M}_{1}(L')\bigotimes\dots\bigotimes\mathcal{M}_{r}(L'))/R,\] where the sum is extended over all finite extensions $L'$ of $L$ and $R$ is the subgroup generated by the following family of elements:\\
If $L\subset K\subset E$ is a tower of finite field extensions and we have elements $x_{i}\in\mathcal{M}_{i}(E)$ for some $i\in\{1,\dots,r\}$ and $x_{j}\in\mathcal{M}_{j}(K)$, for every $j\neq i$, then
    \[x_{1}\otimes\dots\otimes \Tr_{E/K}(x_{i})\otimes\dots\otimes x_{r}-\res_{E/K}(x_{1})\otimes\dots\otimes x_{i}\otimes\dots\otimes\res_{E/K}(x_{r})\in R\]
\end{defn} The relation in $R$ is known as the projection formula.
Using the functoriality properties of the local symbol at each closed point $P\in C$ (prop. 1.5.5. in \cite{IR}), we obtain a complex  \[(\mathcal{M}\bigotimes^{M}\mathbb{G}_{m})(\eta_{C})\stackrel{(.;.)_{C}}{
\longrightarrow}\bigoplus_{P\in C}\mathcal{M}(k)
\stackrel{\sum_{P}}{\longrightarrow}\mathcal{M}(k).\]
Namely, if $C'$ is a smooth complete curve over $k$ with function field $k(C')\supset k(C)$ and we have a section $g\in\mathcal{M}(\eta_{C'})$ and a function $f\in k(C')^{\times}$, then we define
\[(g;f)_{C}=(\sum_{\lambda(P')=P}(g;f)_{P'})_{P}\in\bigoplus_{P\in C}\mathcal{M}(k),\] where $\lambda:C'\to C$ is the finite covering induced by the inclusion $k(C)\subset k(C')$.  \\
We will denote this complex by
$(\underline{\underline{C}})$ and its homology by $H(\underline{\underline{C}})$.
We consider the reciprocity functor $\underline{CH}_{0}(C)$. Notice that the existence of a $k$-rational point $P_{0}\in C(k)$ yields a decomposition of reciprocity functors $\underline{CH}_{0}(C)\simeq\underline{CH}_{0}(C)^{0}\oplus \Z$. We make the following assumption on the $K$-group $T(\mathcal{M},\underline{CH}_{0}(C))$.
\begin{ass}\label{AS1}
Let $\mathcal{M}$ be a reciprocity functor. Let $g\in \mathcal{M}(\eta_{C})$, $h\in \underline{CH}_{0}(C)(\eta_{C})$ and $f\in k(C)^{\times}$. Let $P\in C$ be a closed point of $C$. Assume that the local symbol $(g\otimes h;f)_{P}\in T(\mathcal{M},\underline{CH}_{0}(C))(k)$ vanishes at every point $P$ such that $s_{P}(h)=0$.
\end{ass}
In the next section we will give examples where the assumption \ref{AS1} is satisfied.
\begin{prop}\label{Phi} Let $\mathcal{M}$ be a reciprocity functor over $k$ satisfying \ref{AS1}. Then there is a  well defined map
\begin{eqnarray*}\Phi:&&(\bigoplus_{P\in C}\mathcal{M}(k))/\img((.;.)_{C})\to T(\mathcal{M},\underline{CH}_{0}(C))(k)\\
&&(a_{P})_{P\in C}\to
\sum_{P\in C}a_{P}\otimes [P].
\end{eqnarray*}
\end{prop}
\begin{proof}  First, we immediately observe that if $P\in C$ is any closed point of $C$, then the map $\phi_{P}:\mathcal{M}(k)\to T(\mathcal{M},\underline{CH}_{0}(C))(k)$ given by $a\to
a\otimes [P]$ is well defined. In particular, the map
\[\Phi=\sum_{P}\phi_{P}:\bigoplus_{P\in C}\mathcal{M}(k)\to T(\mathcal{M},\underline{CH}_{0}(C))(k)\] is well defined. Let $D$ be a smooth complete curve over $k$ with generic point $\eta_{D}$ and assume there is a finite covering $\lambda:D\to C$. Let $g\in \mathcal{M}(\eta_{D})$ and $f\in k(D)^{\times}$ be a function. For every closed point $P\in C$  we consider the element $\displaystyle(a_{P})_{P}\in\bigoplus_{P\in C}\mathcal{M}(k)$ such that $a_{P}=(g;f)_{P}$.
We are going to show that $\displaystyle\Phi(\sum_{P\in C}
(g;f)_{P})=0$. \\
First, we treat the case $D=C$ and $\lambda=1_{C}$.
The element $g\in \mathcal{M}(\eta_{C})$ admits a modulus $\mathfrak{m}$ with support $S$.
We consider the zero-cycle $h=[\eta_{C}]\in \underline{CH}_{0}(C)(\eta_{C})$. Notice that for a closed point $P\in C$, the specialization map $s_{P}:\underline{CH}_{0}(C)(\eta_{C})\to \underline{CH}_{0}(C)(k)$ has the property $s_{P}(h)=[P]$.
We are going to show that for every $P\in C$, it holds  \[\Phi((g;f)_{P})=(g\otimes h;f)_{P}\] and the required property will follow from the reciprocity law of the local symbol. We consider the following cases.
\begin{enumerate} \item Let $P\not\in S$. Then,  \begin{eqnarray*}\Phi((g;f)_{P})=&&
\phi_{P}(\ord_{P}(f)s_{P}(g))=
    \ord_{P}(f)s_{P}(g)\otimes[P]=\ord_{P}(f)s_{P}(g)\otimes s_{P}(h)=\\
&&\ord_{P}(f)s_{P}(g\otimes h)=
    (g\otimes h;f)_{P}.
    \end{eqnarray*}
\item Let $P\in S$ and $f\equiv 1$ mod $\mathfrak{m}$ at $P$. Since $\underline{CH}_{0}(C)\in RF_{0}$, $h$ does not contribute to the modulus, and hence, by 2.3 (b)(ii) we get:
\[\Phi((g;f)_{P})=
    \Phi(0)=0=(g\otimes h;f)_{P}.\]
\item Let now $P\in S$ and $f\in K^{\times}$ be any function. We consider an auxiliary function $f_{P}$ for $f$ at $P$. By the definition of the local symbol, we have:
    \begin{eqnarray*}\Phi((g;f)_{P})=&&
    \phi_{P}(-\sum_{Q\not\in S}\ord_{Q}(f_{P})
    s_{Q}(g))=-\sum_{Q\not\in S}\ord_{Q}(f_{P})
s_{Q}(g)\otimes[P]=\\
&&-\sum_{Q\not\in S}\ord_{Q}(f_{P})
s_{Q}(g)\otimes[Q]+\sum_{Q\not\in S}\ord_{Q}(f_{P})
s_{Q}(g)\otimes([Q]-[P]).
\end{eqnarray*} We observe that we have an equality
\[(g\otimes h;f)_{P}=-\sum_{Q\not\in S}\ord_{Q}(f_{P})s_{Q}(g)\otimes [Q].\] Next, notice that the flat embedding $k\hookrightarrow k(C)$ induces a restriction map $\res_{\eta/k}:CH_{0}(C)\to CH_{0}(C\times\eta_{C})$. Let $h_{0}=\res_{\eta/k}([P])$. Then we clearly have
\begin{eqnarray*}&&\sum_{Q\not\in S}\ord_{Q}(f_{P})s_{Q}(g)\otimes([P]-[Q])=(g\otimes(h_{0}-h);f)_{P}.
\end{eqnarray*}
Since we assumed that the assumption \ref{AS1} is satisfied, we get that this last symbol vanishes. For,  $s_{P}(h-h_{0})=0$.
\end{enumerate}
The general case is treated in a similar way. Namely, if $\lambda:D\to C$ is a finite covering of smooth complete curves over $k$ and $g\in\mathcal{M}(\eta_{D})$, then the local symbol at a closed point $P\in C$ is defined to be $(g;f)_{P}=\sum_{\lambda(Q)=P}(g;f)_{Q}$. Considering the zero cycle $h=[\eta_{D}]\in\underline{CH}_{0}(C)(\eta_{D})$, we can show that \[\Phi_{P}((g;f)_{P})=
(g\otimes h;f)_{P}.\]

\end{proof}
\medskip
From now on we fix a $k$-rational point $P_{0}$ of $C$. We obtain the following corollary.
\begin{cor} The map $\Phi$ of proposition \ref{Phi} induces a map
\begin{eqnarray*}\Phi:&&H(\underline{\underline{C}})\to T(\mathcal{M},\underline{CH}_{0}(C)^{0})(k)\\
&&(a_{P})_{P\in C}\to
\sum_{P\in C}a\otimes ([P]-[P_{0}]).
\end{eqnarray*} which does not depend on the $k$-rational point $P_{0}$.
\end{cor}
\begin{proof} If $(a_{P})_{P\in C}\in H(\underline{\underline{C}})$, then $\sum_{P}a_{P}=0\in\mathcal{M}(k)$ and hence $\sum_{P}a_{P}\otimes [P_{0}]=0\in T(\mathcal{M},\underline{CH}_{0}(C))(k)$. We conclude that if $(a_{P})_{P\in C}\in H(\underline{\underline{C}})$ then $\Phi((a_{P})_{P\in C})\in T(\mathcal{M},\underline{CH}_{0}(C)^{0})(k)$ and clearly the map does not depend on the $k$-rational point $P_{0}$.

\end{proof}
\medskip
\begin{defn} Let $\mathcal{M}_{1},\dots,\mathcal{M}_{r}$ be reciprocity functors over $k$. We consider the geometric $K$-group attached to $\mathcal{M}_{1},\dots,\mathcal{M}_{r}$, \[K^{geo}(k;\mathcal{M}_{1},\dots,\mathcal{M}_{r})=
(\mathcal{M}_{1}\bigotimes^{M}\dots\bigotimes^{M}\mathcal{M}_{r})(k)/R,\] where the subgroup $R$ is generated by the following family of elements:\\
Let $D$ be a smooth complete curve over $k$ with generic point $\eta_{D}$. Let $g_{i}\in\mathcal{M}_{i}(\eta_{D})$. Then each $g_{i}$ admits a modulus $\mathfrak{m}_{i}$. Let
    $\displaystyle\mathfrak{m}=\sup_{1\leq i\leq r}\mathfrak{m}_{i}$ and $S$ be the support of $\mathfrak{m}$. Let $f\in k(D)^{\times}$ be a function such that $f\equiv 1$ mod $\mathfrak{m}$. Then
    \[\sum_{P\not\in S}\ord_{P}(f)s_{P}(g_{1})\otimes\dots\otimes s_{P}(g_{r})\in R.\]
\end{defn}
\begin{notation} The elements of the geometric $K$-group $K^{geo}(k;\mathcal{M}_{1},\dots,\mathcal{M}_{r})$ will be denoted as $\{x_{1}\otimes\dots\otimes x_{r}\}^{geo}$.
\end{notation}
\begin{rem} In the notation of \cite{IR} the group $K^{geo}(k;\mathcal{M}_{1},\dots,\mathcal{M}_{r})$ is the same as the Lax Mackey functor $LT(\mathcal{M}_{1},\dots,\mathcal{M}_{r})$ evaluated at $\Spec k$. (def. 3.1.2. in \cite{IR}). In general the group $T(\mathcal{M}_{1},\dots,\mathcal{M}_{r})(k)$ is a quotient of $K^{geo}(k;\mathcal{M}_{1},\dots,\mathcal{M}_{r})$. In the next section we give some examples where these two groups coincide.
\end{rem}
\begin{prop}\label{Psi} Let $P_{0}$ be a fixed $k$-rational point of $C$. The map \begin{eqnarray*}\Psi:&&K^{geo}(k;\mathcal{M},
\underline{CH}_{0}(C)^{0})
\longrightarrow H(\underline{\underline{C}})\\
&&\{x\otimes([P]-[P_{0}])\}^{geo}\longrightarrow(x_{P'})_{P'\in C},
\end{eqnarray*} with $x_{P'}=
\left\{
\begin{array}{ll}
 x, \; P'=P \\
 -x, \; P'=P_{0}\\
 0,\; \rm{otherwise},
 \end{array}
 \right.$ for $P\neq P_{0}$, is well defined and does not depend on the choice of the $k$-rational point $P_{0}$.
\end{prop}
\begin{proof} We start by defining the map $\Psi_{P_{0}}:\mathcal{M}(k)\otimes \underline{CH}_{0}(C)^{0}(k)\to H(\underline{\underline{C}})$ as in the statement of the proposition. To see that $\Psi_{P_{0}}$ is well defined, let $f\in k(C)^{\times}$. We need to verify that for every $x\in\mathcal{M}(k)$ it holds $\Psi_{P_{0}}(x\otimes\dv(f))=0$. Let $\pi:C\to\Spec k$ be the structure map. Consider the pull back $g=\pi^{\star}(x)\in\mathcal{M}(C)$. Then $g\in\mathcal{M}(\eta_{C})$ has modulus $\mathfrak{m}=0$ and hence for a closed point $P\in C$ we have
$(g,f)_{P}=\ord_{P}(f)s_{P}(\pi^{\star}(x))=\ord_{P}(f)x$. Since $\Psi_{P_{0}}(x\otimes\dv(f))=(\ord_{P}(f)x)_{P\in C}$, we conclude that $\Psi_{P_{0}}(x\otimes\dv(f))\in\img(.,.)_{C}$. \\
Next, notice that $\Psi_{P_{0}}$ does not depend on the base point $P_{0}$. For, if $Q_{0}$ is another base point, then
\begin{eqnarray*}&&\Psi_{Q_{0}}(\{x\otimes([P]-[P_{0}])\}^{geo})=\\
&&\Psi_{Q_{0}}(\{x\otimes([P]-[Q_{0}])\}^{geo})-\Psi_{Q_{0}}(\{x\otimes([P_{0}]-[Q_{0}])
\}^{geo}).
\end{eqnarray*} $\Psi_{Q_{0}}(\{x\otimes([P]-[Q_{0}])\}^{geo})$ gives the element
$x$ at the coordinate $P$ and $-x$ at the coordinate $Q_{0}$, while $-\Psi_{Q_{0}}(\{x\otimes([P_{0}]-[Q_{0}])\}^{geo})$ gives $-x$ at coordinate $P_{0}$ and $x$ at $Q_{0}$. From now on we will denote this map by $\Psi$.
In order to show that $\Psi$ factors through $K^{geo}(k;\mathcal{M},\underline{CH}_{0}(C)^{0})$, we consider a smooth complete complete curve $D$ with generic point $\eta_{D}$. Let $g_{1}\in \mathcal{M}(\eta_{D})$ admitting a modulus $\mathfrak{m}$ with support $S_{D}$ and $g_{2}\in\underline{CH}_{0}(C)^{0}(\eta_{D})$ having modulus $\mathfrak{m}_{2}=0$. Let moreover $f\in k(D)^{\times}$ be a function such that $f\equiv 1$ mod $\mathfrak{m}$. We need to show that
\[\Psi(\sum_{R\not\in S_{D}}\ord_{R}(f)\{s_{R}(g_{1})\otimes s_{R}(g_{2})\}^{geo})=0\in H(\underline{\underline{C}}).\]
Since we assumed the existence of a $k$-rational point $P_{0}$, the group $\underline{CH}_{0}(C)^{0}(\eta_{D})$ is generated by elements of the form $[h]-m[\res_{k(D)/k}(P_{0})]$, where $h$ is a closed point of $C\times k(D)$ having residue field  of degree $m$ over $k(D)$. Using the linearity of the symbol on the last coordinate, we may reduce to the case when $g_{2}$ is of the above form. Notice that $h=\Spec k(E)\hookrightarrow C\times\Spec k(D)$, where $E$ is a smooth complete curve over $k$, and hence $h$ induces two coverings
\[ \xymatrix{
& E\ar[d]_{\mu} \ar[r]^{\lambda}   &   D\\
&  C. \\
}
\]
Let $S_{E}=\lambda^{-1}(S_{D})$.
For a closed point $R\in D$, we obtain an equality:
\begin{eqnarray*}s_{R}([h])=\sum_{\lambda(Q)=R}e(Q/R)[\mu(Q)],
\end{eqnarray*}where $e(Q/R)$ is the ramification index at the point $Q\in E$ lying over $R\in D$. Since $\displaystyle m=[k(E):k(D)]=\sum_{\lambda(Q)=R}e(Q/R)$, we get
\begin{eqnarray*}&&\Psi(
\sum_{R\not\in S_{D}}\ord_{R}(f)\{s_{R}(g_{1})\otimes s_{R}(g_{2})\}^{geo})=\\&&\Psi(\sum_{R\not\in S_{D}}\ord_{R}(f)
\{s_{R}(g_{1})\otimes (\sum_{\lambda(Q)=R}e(Q/R)
[\mu(Q)]-m[P_{0}])\}^{geo})=\\
&&\Psi(\sum_{R\not\in S_{D}}\sum_{\lambda(Q)=R}e(Q/R)\ord_{R}(f)\{s_{R}(g_{1})\otimes ([\mu(Q)]-[P_{0}])\}^{geo})=\\
&&\Psi(\sum_{Q\not\in S_{E}}\ord_{Q}(\lambda^{\star}(f))\{s_{Q}(\lambda^{\star}(g_{1}))
\otimes([\mu(Q)]-[P_{0}])\}^{geo}).
\end{eqnarray*} In the above computation we used the fact that for a closed point $Q\in E$ lying over $R\in D$, it holds $s_{R}(g_{1})=s_{Q}(\lambda^{\star}(g_{1}))$ (this equality follows from prop. 1.3.7.($\mathcal{S}_{2}$) in \cite{IR} and the assumption that the base field $k$ is algebraically closed). \\
We conclude that
\begin{eqnarray*}&&\Psi(\sum_{Q\not\in S_{E}}\ord_{Q}(\lambda^{\star}(f))\{s_{Q}(\lambda^{\star}(g_{1}))
\otimes([\mu(Q)]-[P_{0}])\}_{Q/k})=\\
&&\left\{
    \begin{array}{ll}
      \sum_{\mu(Q)=P}\ord_{Q}(\lambda^{\star}(f))s_{Q}(\lambda^{\star}(g_{1})),\Ac P\neq P_{0}  \\
      -\sum_{P\neq P_{0}}\sum_{\mu(Q)=P}\ord_{Q}(\lambda^{\star}(f))s_{Q}(\lambda^{\star}(g_{1}))
,\Ac P_{0}
    \end{array}
  \right.
\end{eqnarray*}
This last computation completes the argument, after we notice that the reciprocity of the local symbol yields an equality
\[-\sum_{P\neq P_{0}}\sum_{\mu(Q)=P}\ord_{Q}(\lambda^{\star}(f))s_{Q}(\lambda^{\star}(g_{1}))
=(\lambda^{\star}(g_{1});\lambda^{\star}(f))_{P_{0}}.
\]

\end{proof}
\medskip
We make the following assumption on $T(\mathcal{M},\underline{CH}_{0}(C)^{0})$.
\begin{ass}\label{AS2} Let $\mathcal{M}$ be a reciprocity functor. Assume that the $K$-group $T(\mathcal{M},\underline{CH}_{0}(C)^{0})(k)$ coincides with the geometric $K$-group $K^{geo}(k;\mathcal{M},\underline{CH}_{0}(C)^{0})$.
\end{ass}
\begin{thm}\label{iso} Let $\mathcal{M}$ be a reciprocity functor such that the group $T(\mathcal{M},\underline{CH}_{0}(C)^{0})(k)$ satisfies both assumptions \ref{AS1} and \ref{AS2}. Then we have an isomorphism $H(\underline{\underline{C}})\simeq T(\mathcal{M},\underline{CH}_{0}(C)^{0})(k)$.
\end{thm}
\begin{proof} By proposition \ref{Psi} we obtain a homomorphism
\[\Psi:T(\mathcal{M},\underline{CH}_{0}(C)^{0})(k)\to H(\underline{\underline{C}}).\] It is almost a tautology to check that $\Psi$ is the inverse of $\Phi$. Namely,
\begin{eqnarray*}\Phi\Psi(x\otimes([P]-[P_{0}]))=\Phi((x_{P'})_{P'})=
\sum_{P'}x_{P'}\otimes [P']=x\otimes[P]-x\otimes[P_{0}],
\end{eqnarray*} and
\begin{eqnarray*}\Psi\Phi((x_{P})_{P})=\Psi(\sum_{P\in C} x_{P}\otimes([P]-[P_{0}]))=(x_{P})_{P}.
\end{eqnarray*}
Notice that for the last equality, we used the fact that $(x_{P})_{P\in C}\in\ker(\sum_{P\in C})$, and hence at coordinate $P_{0}$ we have $\displaystyle x_{P_{0}}=-\sum_{P\neq P_{0}}x_{P}$.

\end{proof}

\vspace{7pt}
\section{Examples}
In this section we give some examples of reciprocity functors $\mathcal{M}$ such that the $K$-group of reciprocity functors $T(\mathcal{M},\underline{CH}_{0}(C)^{0})$ satisfies the assumptions \ref{AS1} and \ref{AS2}.
\subsection{Homotopy Invariant Nisnevich sheaves with Transfers} We consider the category $\HI_{\Nis}$ of homotopy invariant Nisnevich sheaves with transfers over a perfect field $F$. Let $\mathcal{F}_{1},\dots,\mathcal{F}_{r}\in \HI_{\Nis}$. Then each $\mathcal{F}_{i}$ induces a reciprocity functor $\hat{\mathcal{F}}_{i}\in RF_{1}$ (see example 2.3 in \cite{IR}). Moreover, the associated $K$-group of reciprocity functors $T(\hat{\mathcal{F}}_{1},\dots,\hat{\mathcal{F}}_{r})$ is also in $RF_{1}$. We claim that $T(T(\hat{\mathcal{F}}_{1},\dots,\hat{\mathcal{F}}_{r}),\underline{CH}_{0}(C)^{0})$ satisfies both assumptions of theorem \ref{iso}. The claim follows by the comparison of the $K$-group  $T(T(\hat{\mathcal{F}}_{1},\dots,\hat{\mathcal{F}}_{r}),
\underline{CH}_{0}(C)^{0})(k)$ with the Somekawa type $K$-group $K(k;\mathcal{F}_{1}\dots,\mathcal{F}_{r},\underline{CH}_{0}(C)^{0})$ defined by B.Kahn and T.Yamazaki in \cite{KY} (def.5.1). \\
\begin{rem}\label{prod} If $\mathcal{M}_{1},\dots,\mathcal{M}_{r}$ are reciprocity functors with $r\geq 3$, then F. Ivorra and K. R\"{u}lling in corollary 4.2.5. of \cite{IR} prove that there is a functorial map
\[T(\mathcal{M}_{1},\dots,\mathcal{M}_{r})\to T(T(\mathcal{M}_{1},\dots,\mathcal{M}_{r-1}),\mathcal{M}_{r})\] which is surjective as a map of Nisnevich sheaves. It is not clear whether this map is always an isomorphism which would imply that $T$ is associative and we would call it a product. In the case $\mathcal{F}_{i}\in\HI_{\Nis}$, for every $i\in\{1,\dots,r\}$, associativity holds. In fact, in this case there is an isomorphism of reciprocity functors
\[T(\hat{\mathcal{F}}_{1},\dots,\hat{\mathcal{F}}_{r})\simeq(
\hat{\mathcal{F}_{1}\bigotimes_{\HI_{\Nis}}\dots
\bigotimes_{\HI_{\Nis}}\dots\mathcal{F}_{r}}),\] where $\mathcal{F}_{1}\bigotimes_{\HI_{\Nis}}\dots
\bigotimes_{\HI_{\Nis}}\mathcal{F}_{r}$ is the product of homotopy invariant Nisnevich sheaves with transfers. (see 2.10 in \cite{KY} for the definition of the product and theorem 5.1.8 in \cite{IR} for the isomorphism).
\end{rem}
\begin{notn} By abuse of notation from now on we will write $T(\mathcal{F}_{1},\dots,\mathcal{F}_{r})$ for the $K$-group of reciprocity functors associated to $\hat{\mathcal{F}}_{1},\dots,\hat{\mathcal{F}}_{r}$.
\end{notn}
\medskip
\begin{rem}\label{curvelike} Let $\NST$ be the category of Nisnevich sheaves with transfers. We note here that there is a left adjoint to the inclusion functor $\NST\to \HI_{\Nis}$ which is denoted by $h_{0}^{\Nis}$ (see section 2 in \cite{KY}).
If $U$ is a smooth curve over $F$, then there is
a Nisnevich sheaf with transfers $L(U)$, where $L(U)(V)=\Cor(V,U)$ is the group of finite correspondences for $V$ smooth over $F$, i.e. the free abelian group on the set of closed integral subschemes of $V\times U$ which are finite and surjective over some irreducible component of $V$. Then the corresponding homotopy invariant Nisnevich sheaf with transfers $h_{0}^{\Nis}(U):=h_{0}^{\Nis}(L(U))$ is the sheaf associated to the presheaf of relative Picard groups
\[V\to\Pic(\overline{U}\times V,D\times V),\] where $\overline{U}$ is the smooth compactification of $U$, $D=\overline{U}\setminus U$ and $V$ runs through smooth $F$-schemes. When $U$ is projective we have an isomorphism  $h_{0}^{\Nis}(U)\simeq\underline{CH}_{0}(U)$ (see lemma 11.2 in \cite{KY}). In particular, $\underline{CH}_{0}(C)$ is homotopy invariant Nisnevich sheaf with trasnfers. \\
Let $\mathcal{F}\in\HI_{\Nis}$. If we are given a section $g\in\mathcal{F}(U)$ for some open dense $U\subset C$, then $g$ induces a map of Nisnevich sheaves with transfers
$\varphi:h_{0}^{\Nis}(U)\to\mathcal{F}$ such that
\begin{eqnarray*}\varphi(U):&&h_{0}^{\Nis}(U)(U)\to\mathcal{F}(U)\\
&&[\Delta]\to g,
\end{eqnarray*} where $[\Delta]\in h_{0}^{\Nis}(U)(U)$ is the class of the diagonal. The existence of the map $\varphi$ follows by adjunction, since we have  an obvious morphism
$L(U)\to\mathcal{F}$  in $\NST$.
\end{rem}
\begin{lem}\label{HI} Let $\mathcal{F}_{1},\dots,\mathcal{F}_{r}\in\HI_{\Nis}$ be homotopy invariant sheaves with transfers. Then the $K$-group of reciprocity functors $T(T(\mathcal{F}_{1},\dots,\mathcal{F}_{r}),\underline{CH}_{0}(C)^{0})$ satisfies the assumptions of theorem \ref{iso}.
\end{lem}
\begin{proof} By remark \ref{prod} we get an isomorphism \[T(T(\mathcal{F}_{1},\dots,\mathcal{F}_{r}),
\underline{CH}_{0}(C)^{0})(k)\simeq
T(\mathcal{F}_{1},\dots,\mathcal{F}_{r},\underline{CH}_{0}(C)^{0})(k).\] Moreover, by theorem  5.1.8. in \cite{IR} we get that the groups
$K^{geo}(k;\mathcal{F}_{1},\dots,\mathcal{F}_{r},\underline{CH}_{0}(C)^{0})$ and $T(\mathcal{F}_{1},\dots,\mathcal{F}_{r},\underline{CH}_{0}(C)^{0})(k)$ are equal and they coincide with the Somekawa type $K$-group $K(k;\mathcal{F}_{1},\dots,\mathcal{F}_{r},\underline{CH}_{0}(C)^{0})$. We conclude that assumption \ref{AS2} holds.\\
Regarding the assumption \ref{AS1}, let $g_{i}\in \mathcal{F}_{i}(\eta_{C})$ and $h\in\underline{CH}_{0}(C)^{0}(\eta_{C})$ such that $s_{P}(h)=0$ for some closed point $P\in C$. Let moreover $f\in k(C)^{\times}$. We need to verify that $(g_{1}\otimes\dots\otimes g_{r}\otimes h;f)_{P}=0$. If $g_{i}\in \mathcal{F}_{i,C,P}$, for every $i\in\{1,\dots,r\}$, then
\[(g_{1}\otimes\dots\otimes g_{r}\otimes h;f)_{P}=\ord_{P}(f)
s_{P}(g_{1})\otimes\dots\otimes s_{P}(g_{r})\otimes s_{P}(h)=0.\] Assume $P$ is in the support of $g_{i}$ for some $i\in\{1,\dots,r\}$.\\
We first treat the case when $\mathcal{F}_{i}$ is curve-like (see def. 11.1 in \cite{KY})), for $i=1,\dots,r$.
For such $\mathcal{F}_{i}$ it suffices to consider elements $g_{i}\in \mathcal{F}_{i}(\eta_{C})$ with disjoint supports (proposition 11.11 in \cite{KY}). In this case the claim follows by the explicit computation of the local symbol (lemma 8.5 and proposition 11.6 in \cite{KY}). Namely, if $P\in\supp(g_{i})$, then the local symbol at $P$ is given by the formula
\[(g_{1}\otimes\dots\otimes g_{r}\otimes h;f)_{P}=
s_{P}(g_{1})\otimes\dots\partial_{P}(g_{i},f)\otimes\dots\otimes s_{P}(g_{r})
\otimes s_{P}(h)=0,\]
where $\partial_{P}(g_{i},f)$ is the symbol at $P$ defined in section 4.1 of \cite{KY}.\\
Assume now that $\mathcal{F}_{i}$ is general, for $i=1,\dots,r$. Since $g_{i}\in\mathcal{F}_{i}(\eta_{C})$ and $\mathcal{F}_{i}(\eta_{C})\simeq\lim\limits_{\longrightarrow} \mathcal{F}_{i}(U)$,  there is an open dense subset $U_{i}\subset C$ such that $g_{i}\in\mathcal{F}(U_{i})$, for $i=1,\dots,r$. By remark \ref{curvelike} we get that the sections $g_{i}$ induce morphisms in $\HI_{\Nis}$, $\varphi_{i}:h_{0}^{\Nis}(U_{i})\to\mathcal{F}$. In particular, we get a homomorphism
\[\varphi=\varphi_{1}\otimes\dots\otimes\varphi_{r}\otimes 1:  K(k;h_{0}^{\Nis}(U_{1}),\dots,h_{0}^{\Nis}(U_{r}),
\underline{CH}_{0}(C)^{0})
\to K(k;\mathcal{F}_{1},\dots,\mathcal{F}_{r},\underline{CH}_{0}(C)^{0})\] with the property \[(g_{1}\otimes\dots\otimes g_{r}\otimes h;f)_{P}=\varphi(([\Delta_{1}]\otimes\dots\otimes[\Delta_{r}]\otimes h;f)_{P}).\]
Notice that the latter element vanishes, because $h_{0}^{\Nis}(U_{i})$ is curve-like, for $i=1,\dots,r$ (lemma 11.2(c) in \cite{KY}) and hence $([\Delta_{1}]\otimes\dots\otimes[\Delta_{r}]\otimes h;f)_{P}=0$.

\end{proof}
\begin{cor} Let $\mathcal{F}_{1},\dots,\mathcal{F}_{r}\in\HI_{\Nis}$. Let $\mathcal{M}=T(\mathcal{F}_{1},\dots,\mathcal{F}_{r})$ and let $(\underline{\underline{C}})$ be the local symbol complex associated to $\mathcal{M}$ corresponding to the curve $C$. Then there is a canonical isomorphism \[H(\underline{\underline{C}})\simeq T(\mathcal{F}_{1},\dots,\mathcal{F}_{r},\underline{CH}_{0}(C)^{0})(k).\]
In particular, if $G_{1},\dots,G_{r}$ are semi-abelian varieties over $k$, then
\[H(\underline{\underline{C}})\simeq T(G_{1},\dots,G_{r},\underline{CH}_{0}(C)^{0})(k)\simeq K(k;G_{1},\dots,G_{r},\underline{CH}_{0}(C)^{0}),\] where $K(k;G_{1},\dots,G_{r},\underline{CH}_{0}(C)^{0})$ is the usual Somekawa $K$-group attached to semi-abelian varieties. (\cite{Som}, def. 1.2)
\end{cor}
\vspace{5pt}
\subsection{The $\mathbb{G}_{a}$-Case}
In this subsection we consider reciprocity functors $\mathcal{M}_{1},\dots,\mathcal{M}_{r}$, $r\geq 0$ and set $\mathcal{M}_{0}=\mathbb{G}_{a}$. We consider the $K$-group of reciprocity functors $T(\mathbb{G}_{a},\mathcal{M}_{1},\dots,\mathcal{M}_{r})$.
\begin{lem} The $K$-group $T(T(\mathbb{G}_{a},\mathcal{M}_{1},\dots,\mathcal{M}_{r}),
\underline{CH}_{0}(C)^{0})$ satisfies the assumption \ref{AS1}.
\end{lem}
\begin{proof} We have a functorial surjection  \[T(\mathbb{G}_{a},\mathcal{M}_{1},\dots,\mathcal{M}_{r},
\underline{CH}_{0}(C)^{0})(k)\twoheadrightarrow T(T(\mathbb{G}_{a},\mathcal{M}_{1},\dots,\mathcal{M}_{r}),
\underline{CH}_{0}(C)^{0})(k).\]
The first group vanishes by the main result of \cite{RuYa} (theorem 1.1). Therefore, the second group vanishes as well. In particular, \ref{AS1} is satisfied.

\end{proof}
\begin{lem}\label{Ga} The $K$-group $T(T(\mathbb{G}_{a},\mathcal{M}_{1},\dots,\mathcal{M}_{r}),
\underline{CH}_{0}(C)^{0})$ satisfies the assumption \ref{AS2}.
\end{lem}
\begin{proof} To prove the lemma, it suffices to show that  $K^{geo}(k;T(\mathbb{G}_{a},\mathcal{M}_{1},\dots,\mathcal{M}_{r}),
\underline{CH}_{0}(C)^{0})$ vanishes. \\
\underline{Claim:} There is a well defined local symbol
\[T(\mathbb{G}_{a},\mathcal{M}_{1},\dots,\mathcal{M}_{r})(\eta_{C})\otimes
\underline{CH}_{0}(C)^{0}(\eta_{C})\otimes k(C)^{\times}\to K^{geo}(k;T(\mathbb{G}_{a},\mathcal{M}_{1},\dots,\mathcal{M}_{r}),
\underline{CH}_{0}(C)^{0})\] satisfying the unique properties (1)-(3) of section 2.2.  \\
To have a well defined local symbol following Serre (\cite{Ser}), we need for every closed point $P\in C$ the natural map
\[h:T(\mathbb{G}_{a},\mathcal{M}_{1},\dots,\mathcal{M}_{r})(\mathcal{O}_{C,P})
\otimes
\underline{CH}_{0}(C)^{0}(\mathcal{O}_{C,P})\to T(\mathbb{G}_{a},\mathcal{M}_{1},\dots,\mathcal{M}_{r})(\eta_{C})\otimes
\underline{CH}_{0}(C)^{0}(\eta_{C})\] to be injective. For, if $g_{1}\in T(\mathbb{G}_{a},\mathcal{M}_{1},\dots,\mathcal{M}_{r})(\eta_{C})$, $g_{2}\in\underline{CH}_{0}(C)^{0}(\eta_{C})$, then we say that
$g_{1}\otimes g_{2}$ is regular, if $g_{1}\otimes g_{2}=h(\tilde{g}_{1}\otimes\tilde{g}_{2})$, for some $\tilde{g}_{1}\otimes\tilde{g}_{2}\in T(\mathbb{G}_{a},\mathcal{M}_{1},\dots,\mathcal{M}_{r})(\mathcal{O}_{C,P})
\otimes
\underline{CH}_{0}(C)^{0}(\mathcal{O}_{C,P})$. For such $g_{1}\otimes g_{2}$ we can define
$(g_{1}\otimes g_{2};f)_{P}=\ord_{P}(f)s_{P}(\tilde{g}_{1})\otimes s_{P}(\tilde{g}_{2})$. For non regular $g_{1}\otimes g_{2}$ we define the local symbol using an auxiliary function $f_{P}$ for $f$ at $P$ as usual. (see section 2.2) The symbol $(.;.)_{P}$ is well defined, since there is a unique lifting $\tilde{g}_{1}\otimes\tilde{g}_{2}$ and the unique properties (1)-(3) of section 2.2 are satisfied by the very definition of the group $K^{geo}(k;T(\mathbb{G}_{a},\mathcal{M}_{1},\dots,\mathcal{M}_{r}),
\underline{CH}_{0}(C)^{0})$. Therefore to prove the claim, it suffices to show the injectivity of $h$. \\ Note that we have an equality
$\underline{CH}_{0}(C)^{0}(\mathcal{O}_{C,P}):=\underline{CH}_{0}(C\times k(\Spec(\mathcal{O}_{C,P})))^{0}=\underline{CH}_{0}(C)^{0}(\eta_{C})$. Moreover, the map $T(\mathbb{G}_{a},\mathcal{M}_{1},\dots,\mathcal{M}_{r})(\mathcal{O}_{C,P})\to
 T(\mathbb{G}_{a},\mathcal{M}_{1},\dots,\mathcal{M}_{r})(\eta_{C})$ is injective  by the injectivity condition of reciprocity functors. Next, notice that $T(\mathbb{G}_{a},\mathcal{M}_{1},\dots,\mathcal{M}_{r})$ becomes a reciprocity functor of either $\Q$ or $\mathbb{F}_{p}$-vector spaces, depending on whether $\ch F$ is $0$ or $p>0$. Setting $\kappa=\Q$ or $\Z/p$ depending on the case we have,
\begin{eqnarray*}&&T(\mathbb{G}_{a},\mathcal{M}_{1},\dots,\mathcal{M}_{r})
(\mathcal{O}_{C,P})
\otimes_{\Z} \underline{CH}_{0}(C)^{0}(\mathcal{O}_{C,P})=\\
&&T(\mathbb{G}_{a},\mathcal{M}_{1},\dots,\mathcal{M}_{r})
(\mathcal{O}_{C,P})
\otimes_{\kappa} (\kappa\otimes_{\Z}\underline{CH}_{0}(C)^{0}(\mathcal{O}_{C,P})).
\end{eqnarray*} Since the $\kappa$-module $\kappa\otimes_{\Z}\underline{CH}_{0}(C)^{0}(\mathcal{O}_{C,P})$ is flat, the claim follows. \\
To prove the lemma, we imitate the proof of the vanishing of $T(\mathbb{G}_{a},\mathcal{M}_{1},\dots,\mathcal{M}_{r},
\underline{CH}_{0}(C)^{0})(k)$ in \cite{RuYa}. Namely, let $\{(x_{0},\dots,x_{r}),\zeta\}^{geo}$ be a generator of $K^{geo}(k;T(\mathbb{G}_{a},\mathcal{M}_{1},\dots,\mathcal{M}_{r}),
\underline{CH}_{0}(C)^{0})$. Since $k$ is algebraically closed, we may assume $\zeta=[P_{0}]-[P_{1}]$, for some closed points $P_{0},P_{1}\in C$. Then we can show that \[\{(x_{0},\dots,x_{r}),\zeta\}^{geo}=\sum_{P\in C}((x_{0}g\otimes \res_{k(C)/k}(x_{1})\dots\otimes \res_{k(C)/k}(x_{r}))\otimes \eta_{C};f)_{P}=0,\] where $f\in k(C)^{\times}$ is a function such that $\ord_{P_{0}}(f)=1$ and $\ord_{P_{1}}(f)=-1$ and $g\in k(C)^{\times}$ is obtained using the exact sequence
\[\Omega_{k(C)/k}^{1}\longrightarrow\bigoplus_{P\in C}\frac{\Omega^{1}_{k(C)/k}}{\Omega^{1}_{C,P}}
\stackrel{\sum\Res_{P}}{\longrightarrow}k\longrightarrow 0.\] For more details on the above local symbol computation we refer to section 3 in \cite{RuYa}. In particular, we refer to 3.2 and 3.4 for the choice of the functions $f,g\in k(C)^{\times}$.

\end{proof}
\begin{cor} Let $\mathcal{M}_{1},\dots,\mathcal{M}_{r}$ be reciprocity functors. Let $\mathcal{M}=T(\mathbb{G}_{a},\mathcal{M}_{1},\dots,\mathcal{M}_{r})$. Then for any smooth complete curve $C$ over $k$, $H(\underline{\underline{C}})=0$. In particular, if $\ch F=0$, the complex
$\displaystyle\Omega_{k(C)}^{n+1}\stackrel{\Res_{P}}{\longrightarrow}\bigoplus_{P\in C}\Omega^{n}_{k}
\stackrel{\sum_{P}}{\longrightarrow}\Omega_{k}^{n}$ is exact.
\end{cor}
\begin{proof} When $\ch F=0$, Ivorra and R\"{u}lling showed an isomorphism of reciprocity functors $\theta:\Omega^{n}\simeq T(\mathbb{G}_{a},\mathbb{G}^{\times n})$ (Theorem 5.4.7 in \cite{IR}). Moreover, the complex $(\underline{\underline{C}})$ factors through $\Omega^{n+1}_{k(C)}$.

\end{proof}
\vspace{3pt}
\section{The non-algebraically closed Case}
In order to prove theorem \ref{iso}, we made the assumption that the curve $C$ is over an algebraically closed field $k$. The reason this assumption was necessary is that for a general reciprocity functor $\mathcal{M}$ the local symbol at a closed point $P\in C$ does not have a local description, but rather depends on the other closed points. Namely, if $P$ is in the support of the modulus $\mathfrak{m}$ corresponding to a section $g\in\mathcal{M}(\eta_{C})$, then we have an equality
\[(g;f)_{P}=-\sum_{Q\not\in S}\ord_{Q}(f_{P})\Tr_{Q/k}(s_{Q}(g)),\] where $f_{P}$ is an auxiliary function for $f$ at $P$. If for some reciprocity functor $\mathcal{M}$ we have a local description $(g;f)_{P}=\Tr_{P/k}(\partial_{P}(g;f)))$, where $\partial_{P}(g;f)\in\mathcal{M}(P)$, for every $P\in C$, then we can obtain a complex $(\underline{\underline{C}})'$
\[(\mathcal{M}\bigotimes^{M}\mathbb{G}_{m})(\eta_{C})
\stackrel{\partial_{C}}{\longrightarrow}\bigoplus_{P\in C}\mathcal{M}(P)\stackrel{\sum_{P}\Tr_{P/k}}{\longrightarrow}\mathcal{M}(k).\]
For such a reciprocity functor $\mathcal{M}$, assuming the existence of a $k$-rational point $P_{0}\in C(k)$, we can have a generalization of theorem \ref{iso} for the complex $(\underline{\underline{C}})'$ by imposing the following two stronger conditions on $\mathcal{M}$. Namely,
we make the following assumptions.
\begin{ass}\label{AS3} Let $\mathcal{M}$ be a reciprocity functor for which we have a local description of the symbol $(g;f)_{P}=\Tr_{P/k}(\partial_{P}(g;f))$. Let $\lambda:D\to C$ be a finite morphism. Assume that for every $h\in\underline{CH}_{0}(C)(\eta_{D})$ and every closed point $P\in C$ we have an equality \[(g\otimes h;f)_{P}=\Tr_{P/k}(\partial_{P}(g,f)\otimes s_{P}(h)).\]
\end{ass}
\begin{ass}\label{AS4} We assume that for every finite extension $L/k$ we have an equality \[K^{geo}(L;\mathcal{M},\underline{CH}_{0}(C)^{0})\simeq T(\mathcal{M},\underline{CH}_{0}(C)^{0})(L).\]
\end{ass}
\begin{notn} If $E/L$ is a finite extension and $x\in\mathcal{M}(k)$, we will denote $x_{E}=\res_{E/L}(x)$.
\end{notn}
\begin{thm}\label{nonk} Let $\mathcal{M}$ be a reciprocity functor that satisfies the assumptions \ref{AS3} and \ref{AS4}. Then we have an isomorphism \begin{eqnarray*} \Phi':&&H(\underline{\underline{C}}')\stackrel{\simeq}{\longrightarrow} T(\mathcal{M},\underline{CH}_{0}(C)^{0})(k)\\
&&(a_{P})_{P\in C}\to\sum_{P\in C}\Tr_{P/k}(a_{P}\otimes ([P]-P_{0,k(P)})).
\end{eqnarray*}
\end{thm}
\begin{proof}
We start by considering the map
\begin{eqnarray*} \Phi':&&(\bigoplus_{P\in C}\mathcal{M}(P))/\img\partial_{C}\to T(\mathcal{M},\underline{CH}_{0}(C))(k)\\
&&(a_{P})_{P\in C}\to\sum_{P\in C}\Tr_{P/k}(a_{P}\otimes [P]).
\end{eqnarray*}
The map $\Phi'$ is well defined because of the assumption \ref{AS3}. Restricting to $H(\underline{\underline{C}}')$, we obtain the map of the proposition. Moreover, we can consider the map
\begin{eqnarray*}\Psi':&&T(\mathcal{M},\underline{CH}_{0}(C)^{0})(k)\to H(\underline{\underline{C}}')\\&&
Tr_{L/k}(x\otimes([Q]-[L(Q):L][P_{0,L}]))\to (x_{P'})_{P'\in C},
\end{eqnarray*}
with $x_{P}=\Tr_{L(Q)/k(P)}(x)$, $x_{P_{0}}=-\Tr_{L(Q)/k}(x)$ and $x_{P'}=0$ otherwise.  Here $L/k$ is a finite extension, $x\in\mathcal{M}(L)$, $Q$ is a closed point of $C\times L$ having residue field $L(Q)$ that projects to $P\in C$ under the map $C\times L\to C$ with $P\neq P_{0}$. We denote by $k(P)$ the residue field of $P$.
The map $\Psi'$ will be well defined (using the same argument as in proposition \ref{Psi}), as long as we check the following:
\begin{enumerate}\item If $k\subset L\subset E$ is a tower of finite extensions and we have elements $x\in\mathcal{M}(L)$, $y\in\underline{CH}_{0}(C)^{0}(E)$, then $\Psi'(\Tr_{L/k}(x\otimes\Tr_{E/L}(y)))=
\Psi'(\Tr_{E/k}(x_{E}\otimes y))$.
\item If $x\in\mathcal{M}(E)$, $y\in\underline{CH}_{0}(C)^{0}(L)$, then
$\Psi'(\Tr_{L/k}(\Tr_{E/L}(x)\otimes y))=
\Psi'(\Tr_{E/k}(x\otimes y_{E}))$.
\end{enumerate}
For (1) we can reduce to the case when $y=[Q]-[E(Q):E]P_{0,E}$, for some closed point $Q$ of $C\times E$ with residue field $E(Q)$. Let $Q'$ be the projection of $Q$ in $C\times L$ and $P$ the projection of $Q'$ in $C$. Notice that we have an equality $\Tr_{E/L}([Q])=[E(Q):L(Q')][Q']$. We compute:
\begin{eqnarray*}\Psi'(\Tr_{E/k}(x_{E}\otimes([Q]-[E(Q):E][P_{0,E}])))=
&&\left\{
    \begin{array}{ll}
      \Tr_{E(Q)/k(P)}(x), \;\Ac\;P \\
      -\Tr_{E(Q)/k}(x), \;\Ac\;P_{0}
    \end{array}
  \right.
\end{eqnarray*}
\begin{eqnarray*}&&\Psi'(\Tr_{L/k}(x\otimes
\Tr_{E/L}([Q]-[E(Q):E][P_{0,E}])))=\\&&
\Psi'(\Tr_{L/k}(x\otimes[E(Q):L(Q')]([Q']-[L(Q'):L][P_{0,L}])))\\
&&\left\{
    \begin{array}{ll}
      [E(Q):L(Q')]\Tr_{L(Q')/k(P)}(x), \;\Ac\;P \\
      -[E(Q):L(Q')]\Tr_{L(Q')/k}(x), \;\Ac\;P_{0}
    \end{array}
  \right.
\end{eqnarray*} The claim then follows by the fact that \[ \Tr_{E(Q)/k(P)}(x)=\Tr_{L(Q')/k(P)}\Tr_{E(Q)/L(Q')}(x)=
[E(Q):L(Q')]\Tr_{L(Q')/k(P)}(x).\]
For (2) we can again reduce to the case when $y=[Q]-[L(Q):L][P_{0,L}]$ for some closed point $Q$ of $C\times L$ with residue field $L(Q)$. Notice that we have an equality
\[[Q]_{E}=\sum_{Q'\to Q}e(Q'/Q)[Q'],\] where the sum extends over all closed points $Q'$ of $C\times E$ such that $Q'$ projects to $Q$. We compute
\begin{eqnarray*}\Psi'(\Tr_{L/k}(\Tr_{E/L}(x)\otimes y))=&&
\left\{
         \begin{array}{ll}
           \Tr_{L(Q)/k(P)}(\Tr_{E/L}(x)_{L(Q)}), \;\Ac\;P \\
            -\Tr_{L(Q)/k}(\Tr_{E/L}(x)_{L(Q)}), \;\Ac\;P_{0}
              \end{array}
                 \right.\\=
&&\left\{
         \begin{array}{ll}
           \Tr_{L(Q)/k(P)}(\sum_{Q'\to Q}e(Q'/Q)\Tr_{E(Q')/L(Q)}(x_{E(Q')})), \;\Ac\;P \\
            -\Tr_{L(Q)/k}(\sum_{Q'\to Q}e(Q'/Q)\Tr_{E(Q')/L(Q)}(x_{E(Q')})), \;\Ac\;P_{0}
              \end{array}
                 \right.\\=
&&\left\{
         \begin{array}{ll}
           \sum_{Q'\to Q}e(Q'/Q))\Tr_{E(Q')/k(P)}(x_{E(Q')}), \;\Ac\;P \\
            -\sum_{Q'\to Q}e(Q'/Q))\Tr_{E(Q')/k}(x_{E(Q')}), \;\Ac\;P_{0}.
              \end{array}
                 \right.
\end{eqnarray*} The equality $\displaystyle\Tr_{E/L}(x)_{L(Q)}=\sum_{Q'\to Q}e(Q'/Q)\Tr_{E(Q')/L(Q)}(x_{E(Q')})$ follows from remark 1.3.3, property (MF1) in \cite{IR} if we set $\varphi:\Spec E\to\Spec L$ and $\psi:\Spec L(Q)\to\Spec L$. On the other hand we have
\begin{eqnarray*}\Psi'(\Tr_{E/k}(x\otimes y_{E}))=&&
\Psi'(\sum_{Q'\to Q}e(Q'/Q))\Tr_{E/k}(x\otimes ([Q']-[L(Q):L][P_{0,E}]))\\=
&&\left\{
         \begin{array}{ll}
           \sum_{Q'\to Q}e(Q'/Q))\Tr_{E(Q')/k(P)}(x_{E(Q')}), \;\Ac\;P \\
            -\sum_{Q'\to Q}e(Q'/Q))\Tr_{E(Q')/k}(x_{E(Q')}), \;\Ac\;P_{0}
              \end{array}
                 \right.
\end{eqnarray*}
Next we need to show that the maps $\Phi'$, $\Psi'$ are each other inverses.
It is immediate that the composition $\Psi'\Phi'$ is the identity map. For the other composition, we consider an element
$x\otimes ([Q]-[L(Q):L][P_{0,L}])\in T(\mathcal{M},\underline{CH}_{0}(C)^{0})(L)$. If $L(Q)$ is the residue field of $Q$, then $Q$ induces an $L(Q)$-rational point $\widetilde{Q}$ of $C\times L(Q)$. Then we have an equality $\Tr_{L(Q)/L}([\widetilde{Q}])=[Q]$. By the projection formula we get an equality
\[x\otimes([Q]-[L(Q):L][P_{0,L}])=\Tr_{L(Q)/L}(x_{L(Q)}\otimes(
[\widetilde{Q}]-[P_{0,L(Q)}])),\] we are therefore reduced to the case $L(Q)=L$. Then we have
\begin{eqnarray*}\Phi'\Psi'(\Tr_{L/k}(x\otimes([Q]-[P_{0,L}])))=&&
\Tr_{P/k}(\Tr_{L/P}(x)\otimes ([P]-[P_{0,k(P)}]))=\\
&&\Tr_{P/k}\Tr_{L/P}(x\otimes\res_{L/P}([P]-[P_{0,k(P)}]))=\\
&&\Tr_{L/k}(x\otimes[Q]-[P_{0,L}]).
\end{eqnarray*} This completes the proof of the theorem.

\end{proof}
\begin{rem} We note here that for the algebraically closed field case if instead of the assumption \ref{AS1}, we had made the stronger assumption \ref{AS3}, the proof of the proposition \ref{Phi} would have become simpler. The only reason we used \ref{AS1} is that in general the problem of computing the symbol $(g;f)_{P}$ locally is rather hard and is known only in very few cases, namely for homotopy invariant Nisnevich sheaves with transfers, as the next example indicates.
\end{rem}
\begin{exmp} Let $k\in\mathcal{E}_{F}$ be any perfect field. Let $\mathcal{F}_{1},\dots,\mathcal{F}_{r}$ be homotopy invariant Nisnevich sheaves with transfers. Then as mentioned in the previous section, the main theorem of \cite{KY} gives an isomorphism
\[T(\mathcal{F}_{1},\dots,\mathcal{F}_{r})(L)\simeq K^{geo}(L;\mathcal{F}_{1},\dots,\mathcal{F}_{r})\simeq K(L;\mathcal{F}_{1},\dots,\mathcal{F}_{r}),\] where  $K(L;\mathcal{F}_{1},\dots,\mathcal{F}_{r})$ is the Somekawa type $K$-group (\cite{KY},def. 5.1) and $L/k$ is any finite extension. In particular, let $C$ be a smooth, complete and geometrically connected curve over $k$ and $P\in C$ be  a closed point. As in the proof of lemma \ref{HI}, we can reduce to the case when $\mathcal{F}_{i}$ is curve-like, for $i=1,\dots,r$. To describe the local symbol, it therefore suffices to consider sections  $g_{i}\in\mathcal{F}_{i}(\eta_{C})$ with disjoint supports. In this case, if $P$ is in the support of $g_{i}$ for some $i\in\{1,\dots,r\}$ and
$f\in k(C)^{\times}$ is a function, then we have the following explicit local description of $(g_{1}\otimes\dots\otimes g_{r};f)_{P}$.
\[(g_{1}\otimes\dots\otimes g_{r};f)_{P}=\Tr_{P/k}(s_{P}(g_{1})\otimes\dots\otimes\partial_{P}(g_{i},f)
\otimes\dots\otimes s_{P}(g_{r})).\] Moreover, $\underline{CH}_{0}(C)^{0}$ is itself a homotopy invariant Nisnevich sheaf with transfers. Namely, $\underline{CH}_{0}(C)^{0}\in RF_{0}$ and hence the above formula implies  that the assumption \ref{AS3} is satisfied.
\end{exmp}

{99}

\end{document}